\documentclass[letterpaper,12pt]{amsart}
\usepackage{hyperref}
\usepackage{latexsym}
\usepackage{amssymb}
\usepackage[cp850]{inputenc}
\usepackage{epsfig}
\usepackage{amsthm}
\usepackage{amscd}
\usepackage{amsmath}
\usepackage{amsfonts}
\usepackage{graphics}
\usepackage[all]{xy}
\usepackage{dsfont}

\newtheorem{theorem}{Theorem}
\newtheorem{lemma}[theorem]{Lemma}
\newtheorem{corollary}[theorem]{Corollary}
\newtheorem{prop}[theorem]{Proposition}

\newtheorem*{theorem*}{Theorem}
\newtheorem*{corollary*}{Corollary}
\theoremstyle{definition}

\newtheorem{remark}[theorem]{Remark}
\newtheorem{definition}[theorem]{Definition}
\newtheorem*{remark*}{Remark}

\newtheorem{observation}[theorem]{Observation}
\newtheorem*{definition*}{Definition}
\newtheorem*{example*}{Example}

\numberwithin{theorem}{section}

\newcommand{\BC}{\mathbb C} 
\newcommand{\BR}{\mathbb R} 
\newcommand{\BN}{\mathbb N} \newcommand{\BQ}{\mathbb Q}
 \newcommand{\BZ}{\mathbb Z}

\newcommand{\CA}{\mathcal A} \newcommand{\CB}{\mathcal B}
\newcommand{\CC}{\mathcal C} 
\newcommand{\CE}{\mathcal E} \newcommand{\CF}{\mathcal F}
 \newcommand{\CH}{\mathcal H}
\newcommand{\CI}{\mathcal I} 
 
\newcommand{\CM}{\mathcal M} \newcommand{\CN}{\mathcal N}

\newcommand{\CS}{\mathcal S} \newcommand{\CT}{\mathcal T}
 
 \newcommand{\CX}{\mathcal X}
\newcommand{\CY}{\mathcal Y} \newcommand{\CZ}{\mathcal Z}
\newcommand{\aut}{\textup{Aut}(F_n)}

\newcommand{\nid}{\noindent}

\newcommand{\tup}{\textup}

\DeclareMathOperator{\prm}{\BR_{>0}^m}
\DeclareMathOperator{\spc}{\CB(E,R)}
\DeclareMathOperator{\spco}{\CB_0(E,R)}

\newcommand{\comment}[1]{}

\begin{document}
\title    {The spectra of polynomial equations with varying exponents}
\author   {Asaf Hadari}
\date{\today}
\begin{abstract}
We study the dependence of solutions of equations of the form $a_0 + a_1 z^{\ell_1} +  \ldots + a_m z^{\ell_m} = 0$, on the exponents $\ell_1, \ldots, \ell_m$. We apply  our results to equations that appear in graph theory,  the theory of $3$-manifolds fibering over the circle, and the theory of free-by-cyclic groups. In particular, we provide descriptions of the spectra of the Alexander polynomial of a fibered $3$-manifold, Teichm\"uller polynomials associated to such a manifold or to a free by cyclic group, and the family of characteristic polynomials of a fixed directed graph with varying edge lengths. \end{abstract}
\maketitle

\vspace {10mm}

\section{Introduction} \label{Int}

Let $p(z) = a_0 + a_1 z^{\ell_1} +  \ldots + a_m z^{\ell_m} $ be a polynomial. It is a classical problem to study how the roots of $p$ depend on the coefficients $a = (a_0, \ldots, a_m)$. In this paper we study a related question: how do the roots of $p$ depend on the exponents $\ell = (\ell_1, \ldots, \ell_m)$?

 An initial obstruction to studying this question is that it does not make sense for non-integer exponents without choosing a branch of the logarithm function. To avoid this problem, we write $z = e^w$  and turn our attention to studying equations of the form:

$$a_0 + \sum a_i e^{\ell_i w} = 0$$

\nid We call a function of the form $Q(w, \ell) = a_0 + \sum a_i e^{\ell_i w}$ a \emph{poly-exponential}. \\

 The motivation for studying the solutions of poly-exponential equations comes from an ever growing body of poly-exponentials in different areas of mathematics whose roots are fundamentally interesting objects.  Among these, are poly-exponentials associated to:
\begin{enumerate}
\item The classical multivariable Alexander Polynomial and McMullen's Teichm\"uller polynomial, which appear in the theory of fibered $3$-manifolds. 
\item The generalizations of the Teichm\"uller polynomial to the $\tup{Out}(F_n)$ setting, given by Dowdall, Kapovich, Leininger, and separately by Algom-Kfir, Hironaka and Rafi. 
\item The Perron polynomial of a directed graph, which was studied by McMullen.  
\end{enumerate} 

\nid The above examples all follow a similar pattern. Each of them is a polynomial in several variables, say $P \in \BC[X_1, \ldots, X_m]$. Each of them is associated to a family of objects parametrized by a vector $\ell = (\ell_1, \ldots, \ell_m)$, and in each of the examples, the polynomial $P(z^{\ell_1}, \ldots, z^{\ell_m})$ is a polynomial related to that object (for integer values of $\ell_i$). This kind of substitution is called a \emph{specialization}, and we denote  $P(z^{\ell_1}, \ldots, z^{\ell_m})= P_\ell$. 

For instance, the Perron polynomial $P$ is associated to a directed graph $\Gamma$ with $m$ edges. For any $\ell = (\ell_1, \ldots, \ell_m) \in \BN^m$, we can form a new graph by sub-dividing the $i^{th}$ edge into $\ell_i$ edges. The solutions of $P_\ell = 0$ are precisely the roots of the characteristic polynomial of this graph.  \\

The Teichm\"uller polynomial, its generalizations, and the Perron polynomial have an important property that makes one of their zeroes amenable to study. For all of these, there is a cone $\CC$ in the parameter space such that for every primitive integral $\ell \in \CC$, the polynomial $P_\ell$ has a positive real root of multiplicity $1$, whose absolute value is strictly greater than that of all the other roots. Call this number $\rho(\ell)$. We say that such a polynomial is \emph{Perron-Frobenius} (In analogy to the well known Perron-Frobenius Theorem), and we call $\rho(\ell)$ the \emph{Perron-Frobenius root}. 

 Using the inverse function theorem, it is simple to see that in such a case, the function $\rho$ can be extended real-analytically to a degree $-1$ homogeneous function on $\CC$. By other considerations, in all of the above cases, $\log \rho$ is a convex function, and for any $p \in \partial \CC$, $\lim_{\ell \to p} \rho(\ell) = \infty$. \\

The main theme we concern ourselves in this paper is: to what extent do the above results about the Perron-Frobenius root generalize? Do they extend to non Perron-Frobenius polynomials? This would be interesting in the case of the Alexander polynomial, which is not generally Perron-Frobenius. There are no known theorems about how its roots vary. Do the results for Perron-Frobenius roots hold for the other roots? Are there other continuous roots? Do any other roots go to infinity at the boundary of some cone? This would be interesting, for example, for the Perron polynomial. It is a well known idea in graph theory that many of the eigenvalues of the adjacency matrix have important effects on the structure of a graph.  \\

\nid \textbf{Notation.}  For the rest of the paper, we fix $Q(w, \ell) = a_0 + \sum_{i = 1}^m a_i e^{\ell_i w}$. We assume that $a_i \in \BR$, and that $a_0 \neq 0$. Let $\CS = \{\ell \in {\BR_{>0}}^m | \ell_1, \ldots, \ell_{m-1} < \ell_m \}$. For reasons that will become apparent later on, unless otherwise stated -  we will restrict ourselves to values $\ell \in \CS$. Given $\ell \in \CS$, we denote by $Q_\ell = Q(\cdot, \ell)$, and  $$\CZ(\ell) = \{w \in \BC | Q_\ell(w) = 0 \}$$ We call $\CZ(\ell)$ the \emph{spectrum of $Q$ at $\ell$}.

\subsection{The continuity of the spectrum}
\begin{definition} A \emph{root} of $Q$ is a function $\psi: \CS \to \BC$ such that for every $\ell$: $\psi(\ell) \in \CZ(\ell)$. 

\end{definition}

Note that under this definition, the Perron-Frobenius root is indeed a root. Our first main result is that the entire spectrum varies continuously. 
\begin{theorem}\label{Theorem1}
There exists a countable collection of continuous roots $\{\psi_i \}_i$ such that for every $\ell \in \CS$:
$$\CZ(\ell) = \{ \psi_i(\ell)\}_i $$
\end{theorem}

\subsection{Ordering the spectrum} 

The Perron-Frobenius root plays two roles. In addition to being a root, it is also the spectral radius. It is natural to ask, given a poly-exponential $Q$, if the spectral radius varies continuously. More generally, we can attempt to look at the second largest root (in absolute value), or the third, or the smallest, etc. and ask if any of these quantities varies continuously. We begin by defining the functions we are interested in. 

\begin{definition} Given $w_1, w_2 \in \CZ(\ell)$, write $w_1 \cong_\ell w_2$ if $(w_1 -w_2)\ell \in 2 \pi\BZ^m$. That is, $w_1 \cong_\ell w_2$ if $\forall i$, $e^{\ell_i w_1} = e^{\ell_i w_2}$.
\end{definition}

\begin{definition} Let $\ell \in \CS$ have rational coordinates. There are finitely many $\cong_\ell$ classes of roots. Each such class $[w]$ has a finite multiplicity as a zero of $Q$. Denote this multiplicity by  $\mu_{\ell}(w)$. Create an ordered list $$\mathfrak{Re}([w_1]) \leq \ldots \leq \mathfrak{Re}([w_s]) $$
where each equivalence class $[w]$ appears in the list $\mu_\ell(w)$ times. Let $\rho_i(\ell)$ be the $i^{th}$ term from the right (if there exists such a term), and $\lambda_i(\ell)$ be the $i^{th}$ term from the left (if there exists such a term). Note that given $i$, the functions $\rho_i$ and $\lambda_i$ will be defined at all but finitely many rational $\ell$'s. 
\end{definition}

In general, the functions defined above can behave quite terribly. We sum up some of their pathologies in the following theorem. 

\begin{theorem}\label{Theorem2} Suppose $m \geq 2$.
\begin{enumerate}
\item Let $Q$ be a poly-exponential with $m$ terms. Let $\ell \in \CS$. Let $$R = \{i: \rho_i \tup{ cannot be extended continuously at } \ell \}$$ and $$L = \{i: \lambda_i \tup{ cannot be extended continuously } \ell \}$$ Then both sets are infinite.  
\item There exists a poly-exponential $Q$ with $m$ terms, and a co-null set $E \subset \CS$ such that $\rho_1$ is not continuous at any point in $E$. 
\end{enumerate}

\end{theorem}

\subsection{The behavior of roots at $\partial \CS$} 
As it turns out, the topology on $\partial \CS$ inherited from $\BR^m$ is not suitable for studying the behavior of roots as $\ell$ approaches the boundary.  Our strategy will be to introduce a compactification of $\CS / \BR_{>0}$, for which a theorem similar to Theorem \ref{Theorem1} holds. Noting that $\ell \to \CZ(\ell)$ is a degree $-1$ homogeneous function will complete the description we provide. One feature of this description is that it won't differentiate between roots that lie on the unit circle. \\

\nid \textbf{Notation.} Given $\ell = (\ell_1, \ldots, \ell_m) \in \prm$, let $\min(\ell) = \min(\ell_1, \ldots, \ell_m)$, and let $\bar{\ell} = (\ell_m, \ell_m - \ell_1, \ldots, \ell_m - \ell_{m-1})$. 
Define a function $\overline{Q}: \BC \times \BR^m \to \BC$ by setting: $$\overline{Q}(w, \ell) = a_0 e^{\ell_m w} + \sum_{i=1}^{m-1} a_i e^{(\ell_m - \ell_i)w} + a_m = e^{-\ell_mw} Q(-w, \ell)$$
Write $$\CH_- = \{z \in \BC| \mathfrak{Re}(z) < 0 \}$$ $$\CH_+ = \{z \in \BC| \mathfrak{Re}(z) > 0 \}$$  $$\CH_0 = \BC - (\CH_- \cup \CH_+)$$  Let $\CZ_-(\ell) = \CZ(\ell) \cap \CH_-$. Define $\CZ_+$ similarly. 

\smallskip
\nid Let $\iota: \BR^m_{\geq 0} \to [0,\infty^m] \times [0,\infty]^m$ be the degree $0$, $\BR_{>0}$-homogeneous function given by: 

$$\iota(\ell) = (\frac{1}{\min(\ell)} \ell, \frac{1}{\min(\bar{\ell})} \bar{\ell}) $$

\nid Following the convention that $e^{\infty w} = 0$ for any $w\in \CH_-$, we can extend the definition of $\CZ_-$ to $[0,\infty]^m$. 

\begin{definition} Given $(p,q) \in [0,\infty]^m \times [0,\infty]^m$, let $$\CZ(p,q) = \CZ_-(Q, p)  \cup -\CZ_-(\overline{Q}, q)$$
\end{definition}

\nid Let $\BC' = \BC / \CH_0$, and let $z \to [z]$ be the corresponding quotient map.

\begin{theorem}\label{Theorem3}
Let $\{\psi_i \}_i$ be a set of roots, as provided by Theorem \ref{Theorem1}. For each $ i $, let: 

$$\psi_i'  (\ell) = \left\{ \begin{array}{l c} [\psi_i(\ell)]& \psi_i(\ell) \in \CH_0  \\ \left[ \min(\ell) \psi_i \right]&\psi_i(\ell) \in \CH_- \\  \left[ \min(\overline{\ell})\psi_i \right]&\psi_i(\ell) \in \CH_+   \end{array} \right.$$

Given $(p,q) \in [0,\infty]^m \times [0,\infty]^m$, then for every $i$,  $\lim_{\iota(\ell) \to (p,q)} \psi_i'(\ell)$ exists and $$\{\lim_{\iota(\ell) \to (p,q)} \psi_i'(\ell)\}_i  - \{[0] \}= \CZ(p,q)$$

\end{theorem}

\begin{corollary}\label{Corollary1} In the notation above, if $\psi$ is a root, and $\iota(\ell) \to (p,q)$ then one of the following happens. 
\begin{enumerate}
\item There exists $\lambda > 0$ such that:  $$\lim_{\iota(\ell) \to (p,q)} \frac{|\exp(-\psi(\ell)) |}{\exp(\frac{\lambda}{\min(\ell)})} = 1 $$
\item There exists $\lambda > 0$ such that:  $$\lim_{\iota(\ell) \to (p,q)} \frac{|\exp(\psi(\ell)) |}{\exp(\frac{\lambda}{\min(\overline{\ell})})} = 1 $$
\item The root $[\psi]$ is bounded as $\iota(\ell) \to (p,q)$.
\end{enumerate}
\end{corollary}

\begin{remark}\label{Remark1}
This corollary provides new information, even in the case of a Perron-Frobenius root. Namely, it provides the rate at which the root goes to $\infty$ at the boundary.  
\end{remark}
\begin{remark} Suppose $a_i \in \BZ$ for $1 \leq i \leq m$. One feature of Corollary \ref{Corollary1}, is that for every $(p,q)$ on the boundary of $\iota(\CS)$, exactly one of the following holds: 
\begin{enumerate}
\item There is a root $\psi$ such that $|\psi|$ goes to $\infty$ at $(p,q)$. 
\item After discarding infinite exponents, both $p$ and $q$ are cyclotomic polynomials. 
\end{enumerate} 
It can be shown using Proposition \ref{Prop3.2.1} that the set of all $(p,q)$ such that both $p$ and $q$ are cyclotomic is meager, in the sense of the Baire category theorem. 
\end{remark}

\nid \textbf{Organization of the paper}. Section \ref{Section2} discusses the poly-exponentials mentioned in the introduction. While \ref{Section2.1} and \ref{Section2.2} are only introductions, \ref{Section2.3} contains new results. Notably, it provides a new family of poly-exponentials whose relationship to the Perron polynomial is similar to the relationship between the Alexander polynomial and the Teichm\"uller polynomial. Furthermore, it contains results about the set of $\ell$ for which the subdivided graph $\Gamma_\ell$ has properties that are usually seen as rare, such as having a non-diagonalizable adjacency matrix. Section \ref{Section3} contains the proofs of the theorems.  Section \ref{Sec4} contains several explicit examples in which the theorems are applied. \\

\nid \textbf{Acknowledgements.} The author would like to thank Jayadev Athreya, Eriko Hironaka, Thomas Koberda, and Yair Minsky for helpful discussions concerning various incarnations of the results in this paper.

\section{Some important poly-exponentials}\label{Section2}
\subsection{Poly-exponentials in 3-manifold theory}\label{Section2.1}
Let $M$ be a compact $3$-manifold (possibly with boundary) that fibers over the circle $S^1 = \BR / \BZ$. To any such fibration of $M$ we attach several objects: 

\begin{enumerate}
\item A fiber $S$. 
\item The isotopy class of the monodromy: $\phi \in \tup{Mod}(S)$, where $\tup{Mod}(S)$ is the mapping class group of $S$. 
\item A cohomology class $\omega \in H^1(M; \BZ)$ defined by pulling back the form $dx$ from $S^1$.   

\end{enumerate}

\nid The class $\omega \in H^1(M;\BZ)$ is said to be \emph{fibered} if it corresponds to a fibration. In \cite{thurstnorm}, Thurston studied which cohomology classes are fibered. He proved the following theorem. 

\begin{theorem} (\textup{Thurston}, \cite{thurstnorm}) Let $M$ be as above. Then there exists a semi-norm on $H^1(M;\BR)$ with a polyhedral unit norm ball, and a collection of faces $\CF_1, \ldots, \CF_k$ (called fibered faces) of this ball such that the fibered cohomology classes of $M$ are precisely the primitive $\BZ$-points in the interiors of the cones $\BR_{+}\CF_1, \ldots, \BR_{+}\CF_k$.
\end{theorem}

\nid In \cite{thurston}, Thurston proved that the manifold $M$ is hyperbolic if and only if the monodromy $\psi$ is pseudo-Anosov. In particular, if the monodromy of one fibered class is pseudo-Anosov, then the monodromy of every fibered class is pseudo-Anosov.  Furthermore, he showed that in this case the semi-norm described above is actually a norm. 

There are two important poly-exponentials that can be associated to such a manifold $M$. Let $G = \BZ[H_1(M, \BZ) / \tup{torsion}]$.  Given  $P = \sum_g a_g g \in G$, and $\omega \in H^1(M; \BR)$, we can form a Laurent polynomial: $$P_{(\omega)}(t) = \sum_g a_g t^{\omega(g)} $$
In this way, we can naturally associate a poly-exponetial to P, which we'll call $Q_P$. The poly-exponentials discussed below will be of this form. 

\subsubsection{The Teichm\"uller polynomial}\label{Section2.1.1}

\nid One of the most important invariants assigned to a pseudo-Anosov mapping class $\psi$ is its dilatation, $\rho(\psi)$. In \cite{Friedzeta}, Fried investigated the relation between $\rho(\psi)$ and $\omega$. In particular, let $M$ be a compact, oriented, hyperbolic 3-manifold that fibers over the circle, and let $\CF$ be a fibered face. Given a fibered class $\omega \in \BR_+ \CF$, let $\rho(\omega)$ be the dilatation of its monodromy. Fried proved the following theorem. 
\begin{theorem} (\textup{Fried} \cite{Friedzeta})
The function $\Lambda(\omega) = log \rho(\omega)$ can be extended to a real analytic, convex, degree $-1$ homogeneous, function on $\BR_+\CF$. Furthermore, $$\lim_{\omega \to \partial \BR_+ \CF} \Lambda(\omega) = \infty$$
\end{theorem}

\nid In \cite{McMullenP}, McMullen extended this theory, and brought it into the realm of poly-exponentials. To any pseudo-Anosov element $\varphi \in \tup{Mod}(S)$, one can associate an invariant train track $\tau \subset S$ (see \cite{McMullenP} for definitions, and further information). Let $T(\tau)$ denote the $\BZ$-module generated by the edges of $\tau$, modulo the relations: 

$$[e_1] + \ldots + [e_r] = [e_1'] + \ldots + [e_s'] $$
for each vertex $v$ of $\tau$, where $e_1, \ldots e_r$ are the incoming edges at $v$, and $e_1', \ldots, e_s'$ are the outgoing edges. The pseudo-Anosov $\varphi$ induces an action $\varphi^*$ on $Z_1(\tau, \BR) = \tup{Hom}(T(\tau), \BR)$. Let $\xi = \xi_\varphi$ be the characteristic polynomial of this action, then $\xi$ is a Perron-Frobenius polynomial, and its Perron-Frobenius root is equal to the dilatation of $\varphi$.. \\ 

\smallskip

\nid One of the results McMullen proved is the following. 

\begin{theorem} (\textup{McMullen, \cite{McMullenP}}) Let $M$ be a compact, oriented, hyperbolic $3$-manifold, and let $\CF$ be a fibered face. In the notation above, there exists $\Theta \in G$, called the \emph{Teichm\"uller polynomial of $\CF$}, such that for any fibered $\omega \in \BR_+ \CF$, $\exists n \geq 0$ such that $$t^n\Theta_{(\omega)} = \xi_{\varphi(\omega)}$$
where $\varphi(\omega)$ is the monodromy associated to $\omega$.
\end{theorem}

McMullen, in a somewhat different language,  studied the properties of the Perron-Frobenius root of $\Theta_{(\omega)}$,  showing that it has to go to $\infty$ at $\partial \BR_+\CF$ and that its $\log$ extends real-analytically to a convex function. Our results extend this analysis by studying all other roots, as well as providing the rates at which the Perron-Frobenius root goes to $\infty$. 

\subsubsection{The Alexander Polynomial}\label{Section2.1.2}

Let $M$ be as above. Let $\Delta = \Delta_M \in G$ be its (multivariate) Alexander polynomial, originally defined by Fox (see \cite{Fox}) One well known fact about the Alexander polynomial is the following. Suppose $\omega \in H^1(M; \BZ)$ is a fibered cohomology class, with fiber $S$ and monodromy $\varphi: S \to S$. Then up to multiplication by $t^n$ for some $n$, $\Delta_{(\omega)}(t)$ is the characteristic polynomial of the action of $\varphi$ on $H_1(S, \BC)$.

\subsection{Poly-exponentials in $\tup{Out}(F_n)$ theory}\label{Section2.2}
There is a well known, and much studied, connection between mapping class groups of closed, oriented surfaces and $\tup{Out}(F_n)$ - the outer automorphism groups of free groups. Quite often, results concerning mapping class groups can be generalized, with appropriate modifications, to the $\tup{Out}(F_n)$ setting. Several recent papers have focused on carrying out this program for the set of results described in section \ref{Section2.1.1}.

Let $\varphi \in \aut$ be an automorphism of a free group. Form the semi-direct product $\Gamma = F_n \rtimes_\varphi \BZ$. The group $\Gamma$ may be decomposable as a free-by-cyclic semi-direct product (where the free group has finite rank) in more than one way. To each such decomposition, one can associate a homomorphism $\omega \in H^1(\Gamma; \BZ)$, and a monodromy $\varphi_\omega \in \tup{Aut}\big(\tup{Ker}(\omega)\big)$, where $\tup{Ker}(\omega)$ is a finitely generated free group. By analogy to the $3$-manifold case, call such a cohomology class fibered. 

 Bieri, Neumann, and Strebel proved that the there is an open cone in $H^1(\Gamma, \BR)$ whose integral, primitive points were exactly the set of fibered classes (\cite{BNS}). In \cite{DLK} Dowdall, Kapovich, and Leininger were able to show that on a smaller cone, these fibers could be studied geometrically. They restricted to the case where $\varphi$ hyperbolic and fully irreducible (see \cite{} for definitions). They produced a $K(\Gamma, 1)$-space $X_\varphi$ which they called the folded mapping torus of $\varphi$, equipped with a semi flow $\psi$. They then produced an open cone $\CA \subset H^1(\Gamma; \BR)$, such that  any integral, primitive $\omega \in \CA$ is fibered. Furthermore, any such $\omega$ is dual to a graph which is a cross-section of $\psi$, has fundamental group $\tup{Ker}(\omega)$, and a first return map $\varphi_\omega$ which is hyperbolic and fully irreducible. They then proved that the function $\omega \to \mathfrak{h}(\omega)$, assigning to each fibered form the log of the dilatation of $\varphi_\omega$ could be extended real-analytically to a convex, degree $-1$ homogeneous function. 
 
Given the similarities to the $3$-manifold theory situation, it is natural to  ask if there is an analog of the Teichm\"uller polynomial for this setting. This question was answered separately by Dowdall, Kapovich, Leininger in \cite{DLK2}, and by Algom-Kfir, Hironaka, and Rafi in \cite{YEK}.  Let $G = H_1(\Gamma) / \tup{torsion}$. Each group produced an element $\mathfrak{m}, \Theta \in \BZ[G]$, with the property that for every $\omega \in \CA$, the spectral radius of $\mathfrak{m}_{(\omega)}$ and $\Theta_{(\omega)}$ is equal to $\mathfrak{h}(\omega)$. 

As shown in \cite{DLK2}, the polynomial $\Theta$ is a factor of $\mathfrak{m}$, but the two are not equal in general. The polynomials $\mathfrak{m}$ and $\Theta$ have separate important features in the context of our theory of poly-exponentials. Namely, the following properties hold for them.   

\begin{theorem} (\textup{Algom-Kfir, Hironaka, and Rafi}, \cite{YEK}) There is a cone $\CA \subset \CT$ on which $\Theta$ is Perron Frobenius. In this cone, if $\theta' \in \BZ[G]$ satisfies that the spectral radius of $\theta'_{(\omega)}$ is the same as the spectral radius of $\Theta_{(\omega)}$ for every integral $\omega$, then $\Theta$ divides $\theta'$.
\end{theorem}

\begin{theorem}(\textup{Dowdall, Kapovich, and Leininger}, \cite{DLK2}) For any fibered $\omega \in \CA$, the polynomial $\mathfrak{m}_{(\omega)}$ is the characteristic polynomial of the train-track transition matrix of $\varphi_\omega$.
\end{theorem}

\subsection{Poly-exponentials in graph theory}\label{Section2.3}

Let $\Gamma$ be a finite directed graph, possibly with loops and parallel edges. Let $\ell$ be a metric on $\Gamma$, given by a function $\ell: E(\Gamma) \to \BR_{>0}$. If $\ell(e) \in \BN$
 for every edge $e$, it makes sense to construct a graph $\Gamma_\ell$ by subdividing each edge, $e$, into $\ell(e)$ edges, each inheriting the orientation from $e$. Let $A_\ell$ be the adjacency matrix of $\Gamma_\ell$ and let  $\xi_\ell$ be the characteristic polynomial of $A_\ell$. One natural question to ask is - how does $\CZ(\xi_\ell)$ depend on $\ell$? This is analogous to  the well studied problem of how the spectrum of the Laplace-Beltrami operator on a Riemannian manifold depends on the metric. 
 
In \cite{McMPerron}, McMullen defined a  polynomial, which he called the Perron polynomial, in the following way. Suppose $n = \#V(\Gamma)$. Define $A(t, \ell) \in M_n(\BZ[t])$ by setting: 
$$A_{uv}(t, \ell) = \sum_{[e] = (u,v)}t^{\ell(e)}$$
where $[e] = (u,v)$ if $e$ connects $u$ to $v$. The Perron polynomial is equal to $$P(t, \ell) = \det (I - A(t, \ell))$$

\nid McMullen then proved the following theorem: 
\begin{theorem}(\textup{McMullen} \cite{McMPerron}) If $\Gamma$ has at least one closed path, then the smallest positive zero of $P_\ell$ is equal to the spectral radius of $\xi_\ell$.
\end{theorem}

We begin by extending this theorem to apply to the entire spectrum. 

\begin{prop}\label{Proposition4} In the notation above, for any $0 \neq \lambda \in \BC$: $$\xi_\ell(\lambda) = 0 \iff P(\lambda^{-1}, \ell) = 0 $$

\end{prop}

\begin{proof} The adjacency matrix $A_\ell$ acts on the space $W = \BC^{V(\Gamma_\ell)}$, by identifying elements of $W$ as column vectors and matrix multiplication. Let $f \in W$ be a $\lambda$-eigenvector of $A_\ell$. This happens if and only if $\forall u \in V(\Gamma_\ell)$: 

$$\sum_{[e] = (u,v) } f(v) = \lambda f(u) $$

View the set $V(\Gamma)$ as a subset of $V(\Gamma_\ell)$. For every $v \in V(\Gamma_\ell)$, define $\tau(v)$ to be the closest element of $V(\Gamma)$ to which $v$ can be connected by a path in $\Gamma_\ell$. Let $h(v)$ be the length of this shortest path. So, if $v \in V(\Gamma)$, we have that $\tau(v) = v$, and $h(v) = 0$. On the other hand, if $v \notin V(\Gamma)$, then $v$ is contained in a unique edge of $\Gamma$, and $\tau(v)$ is the endpoint of that edge. 

Suppose $v \in V(\Gamma_\ell) - V(\Gamma)$, and $h(v)= 1$. The vertex $v$ is connected to one vertex: $\tau(v)$, by one edge. Plugging into the above equation, we get: $\lambda f(v) = f(\tau(v))$. Applying the same reasoning, for any $v \in V(\Gamma_\ell) - V(\Gamma)$, we have that: 
$\lambda^h f(v) = f(\tau(v)) $ , or in other words $f(v) = \lambda^{-h} f(\tau(v))$.  For $v \in V(\Gamma)$ we get the following equation: 

$$f(v) = \sum_{e \in E(\Gamma), [e] = (v,u)} \lambda^{-\ell(e)} f(u) $$
Let $\phi \in \BC^{\#V(\Gamma)}$ be the function $f|_{V(\Gamma)}$. Then the above equation, taken over all $v$, can be written: $\phi = A(\lambda^{-1},\ell) \phi$. Notice that if this equation has a solution, then it can be extended to an $A_\ell$ eigenvector. In other words, $\xi_\ell(\lambda) = 0$ if and only if $1$ is an eigenvalue of $A(\lambda^{-1}, \ell)$, which happens if and only if $P(\lambda^{-1},\ell) = 0$.
\end{proof}

Note that by replacing $P$ with the corresponding poly-exponential, we can consider values of $\ell$ which do not assign integer lengths to edges. 
\smallskip

We can extend Proposition \ref{Proposition4} further, by providing poly-exponentials whose relationship to $P$ is similar to the relationship between the Teichm\"uller polynomial and the Alexander polynomial. 

Fix a graph automorphism $T: \Gamma \to \Gamma$. For some values of $\ell$ the automorphism $T$ can be extended to a graph automorphism of $\Gamma_\ell$. Let $V_T$ be the set of all such values.  The automorphism $T$ acts on the space $W$ defined above. Confusing $T$ with this action, we get that $T A_\ell T^{-1} = A_\ell$. Fix $\lambda_0$, an eigenvalue of $T$. Every $\lambda_0$ eigenvector for $T$ is an eigenvector for $A_\Gamma$. Let $Z(\ell, \lambda_0)$ be the set of $A_\ell$ eigenvalues that arise in this way. 

\begin{prop}\label{Proposition5} There exists a polynomial $P_{T, \lambda_0}$, such that for every $\ell \in V_T$: 
$$P_{T, \lambda_0}(\lambda^{-1}, \ell) = 0 \iff \lambda \in Z(\ell, \lambda_0)$$

\end{prop}

\begin{proof} 
Let $f$ be a $\lambda$-eigenvector for $A_\ell$ and a $\lambda_0$-eigenvector for $T$. Choose a set of representatives $w_1, \ldots, w_s$ for the $T$ orbits in $V(\Gamma)$. For every $v \in V(\Gamma)$, we get $f(v) = \lambda_0^j w_i$, for some value of $i,j$. Write $w_i = [v]$, and $j = \mu(v)$  For each $w_k$, similarly to the previous proof, we get an equation: 
$$f(w_k) = \sum_{[e] = (w_k, v) } \lambda^{-\ell(e)} \lambda_0^{\mu(v)}f([v]) $$
Writing $\phi = f|_{w_1, \ldots, w_s}$, we get, in matrix notation $\phi = D \phi $, where $D = D(\lambda^{-1}, \ell)$ and $$D_{w_1, w_2} = \sum_{[e] = (w_1, v), [v]= w_2} \lambda^{-\ell(e)}\lambda_0^{\mu(v)}$$
By definition, whenever such a solution exists, we can extend it to a $\lambda$ eigenvector of $A_\ell$ that is also a $\lambda_0$ eigenvector for $T$. Thus, we can take $P_{T, \lambda_0} = \det(I - D)$
\end{proof}
\nid Once again, replacing $P_{T, \lambda_0}$ with the corresponding poly-exponential allows us to take non-integer values of $\ell$.

\subsubsection{Remark on the difference between $P_\ell$ and $\xi_\ell$}\label{Section2.3.1}
In general, $P(t, \ell) \neq \xi_\ell(t) $. One way to see this is to calculate degrees. Indeed, the degree of $P_\ell$ is at most  $$\delta(\ell) := \sum_{v \in V(\Gamma)} \max_{e \in E(\Gamma), [e] = (v,u)} \ell(e) $$

\nid whereas $|\ell|:= \deg(\xi_\ell) = \sum_{e \in E(\Gamma)} \ell(e)$. Thus, if $\Gamma$ has a pair of parallel edges with lengths $\geq 1$  then the polynomials will have different degrees. One reason that the degrees are different is that $0$ may be an eigenvalue of $A_\ell$, but it will never be detected by $P$. Using the same argument as the one in the proof of Proposition \ref{Proposition4}, we see that if $\phi$ is a $0$-eigenvector for $A_\ell$ then $\phi(v) = 0$ for every $v$ such that $h(v) \neq 1$. Thus, the geometric multiplicity of $0$ as an eigenvalue of $A_\ell$ is at most $\#E(\Gamma) - \#V(\Gamma)$. 

We can leverage the difference in degrees, together with Proposition \ref{Proposition4}  to the get the somewhat surprising result that for graphs with many parallel edges,  $A_\ell$ is quite often not diagonalizable and has roots with multiplicity greater than $1$. Indeed, notice that for any non-zero eigenvalue $\lambda$ of $A_\ell$, the geometric multiplicity of $\lambda$ is at most $\#V(\Gamma)$. Thus, we get the following.

\begin{corollary}\label{Corollary2} Let $\Gamma$, $\ell$ be as above. If $$\#V(\Gamma) \big(\delta(\ell) -1 \big) + \#E(\Gamma) <  |\ell| $$
then $A_\ell$ is not diagonizable.  Furthermore, if $$\delta (\ell) +  \#E(\Gamma) - \#V(\Gamma) < |\ell| $$
then $\xi_\ell$ has roots of multiplicity greater than $1$.
\end{corollary}

The characteristic polynomials $\xi_\ell$ are all represented by a poly-exponential as well. Define a \emph{multi-cycle} in a direct graph to be a disjoint collection of cycles. Let $\CC$ be the collection of all multi-cycles. If $\sigma$ is a multi-cycle, let $n(\sigma)$ denote the number of cycles in $\sigma$, and let $|\sigma|$ denote its length. In \cite{CR}, Cvetovi\'c and Rowlinson proved the following theorem. 

\begin{theorem} (\textup{Cvetovi\'c and Rowlinson }, \cite{CR}) In the notation above, 

$$\xi_\ell = t^{|\ell|}\big(1 + \sum_{\sigma \in \CC} {(-1)}^{n(\sigma)} t^{-|\sigma|}\big) $$
\end{theorem}

Since the number of multi-cycles is independent of $\ell$, and their lengths are linear functions of the $\ell_i$'s, we see that there is a poly-exponential as required. Both this poly-exponential and the Perron poly-exponential convey slightly different information - namely in determining multiplicities, and are  both useful. 

\section{Proofs}\label{Section3}

\subsection{Proof of Theorem \ref{Theorem1}}\label{Section3.1}

\begin{proof}

\nid We first set up some notation.  Let $\CX = Q^{-1}(0)$. Note that $\CX$ is an analytic set. Let $\pi_w: \BC \times \BR^m \to \BC$ and $\pi_\ell: \BC \times \BR^m \to \BR^m$ be the obvious projections. Given $A \subset \prm$, let $\CX_A = \CX \cap \pi_\ell^{-1}(A)$. Endow $\BR^m$ with the standard Euclidean metric, and $\BC\times \BR^m$ with the Euclidean metric we get by identifying it in the obvious way with $\BR^{m+2}$. Given $t = (w_0, \ell_0) \in \CX$, and $r > 0$, let $\beta(t,r)$ be the connected component of $$\{(w, \ell) \in \CX : \|\ell - \ell_0 \| < r\}$$ which contains $t$.

The following Proposition is the main technical tool in the proof of Theorem \ref{Theorem1}:

\begin{prop}\label{Prop3.1.1} Fix $A$ relatively compact in $\CS$. There exists constants $C_1, C_2, C_3$ such that for any $t = (\omega, l)\in \CX_A$ and any $r > 0$ with $\beta(t,r) \subset \CX_A$: 
$$\tup{diameter } \pi_w \beta (t, r) < C_1 (|\omega| + 1)^{C_2} e^{C_3 r} $$
\end{prop}

Since the proof of Proposition \ref{Prop3.1.1} is somewhat involved, we present it separately in  \ref{Sec3.1.1}.

\begin{prop}\label{Prop3.1.2} Let $\CY$ be a connected component of an irreducible component of $\CX$. Then $\pi_\ell \CY \cap \CS = \CS$
\end{prop}
\begin{proof} We write: $\CY = \CY_\BC \cap \BC \times \BR^m $ where $\CY_\BC$ is a collection of connected components of irreducible components of $Q^{-1}(0)$ in $\BC \times \BC^m$. The projection map $\pi_\ell$ is analytic, and hence open. The set $\CY_\BC$ is an analytic set. Hence the image of $\CY_\BC$ under this map is open in $\BC^m$, and thus the image of $\CY$ of $\pi_\ell$ is open in $\BR^m$. Now, we show that $\pi_\ell \CY$ is relatively closed in $\CS$. Indeed, let $U \subset \CS \cap \pi_\ell \CY$  be relatively pre-compact in $\CS$. By perhaps restricting to a subset, we can assume that $\pi_\ell^{-1}(U) \cap \CY$ is connected. By Proposition \ref{Prop3.1.1}, $\pi_\ell^{-1}(U) \cap \CY$ is bounded. Thus, given $l \in \overline{U}$, $\exists \{t_i\}_i = \{(\omega_i, l_i) \}_i \subset \CY$ with $l_i \in U$, $l_i \to l$, and $\{w_i \}_i$ bounded. By passing to a subsequence, we can assume $w_i \to w$. By continuity  $(w, l) \in \CY$, and thus $l \in \CS \cup \pi_\ell \CY$. The result follows from the observation that $\CS$ is connected. 

\end{proof}

\begin{corollary}\label{Cor3.1.3} Let $\ell \in \CS$ and let $w \in \CZ(\ell)$, then there exists a continuous root $\psi: \CS \to \BC$ such that $\psi(\ell) = w$.
\end{corollary}

\begin{lemma}\label{Lemma3.1.4} There are at most countably many roots $\rho: \CS \to \BC$. 
\end{lemma}

\begin{proof} For any $A \subset \CS$, pre-compact and open, and any region $K \subset \BC$ , there are only finitely many possible roots $\psi: A \to K$. This follows from the fact that for any $\ell \in \CS$, $\CZ(\ell) \cap K$ is finite, and that $\CX \cap K \times A$ is a finite collection of manifolds. Taking countable exhaustions of $\CS$, $\BC$  gives the result. 
\end{proof}

Theorem \ref{Theorem1} follows directly from Corollary \ref{Cor3.1.3} and Lemma \ref{Lemma3.1.4}. 
\end{proof}
\subsection{Proof of Proposition \ref{Prop3.1.1}}\label{3.1.1}
\begin{proof}
\begin{lemma}\label{Lemma3.1.5} Let $A \subset \CS$ be relatively compact in the topology induced from $\BR^m$. There exists $C = C(A) >0$, such that for any $\ell \in A$, and $\alpha \in \CZ(\ell)$: $-C  < \mathfrak{Re}(\alpha) < C$.
\end{lemma}

\begin{proof}
Fix  $\ell \in \CS$ and assume, without loss of generality that $0< \ell_1 \leq  \ell_2 \leq \ldots \leq \ell_{m-1} < \ell_m$. Let $M = \max_i (|a_i|)$. For $x \in \BC$ whose real part is sufficiently large, $|a_me^{(\ell_m - \ell_{m-1}) x}| > (m-1)M$, and thus $|a_m e^{\ell_m x}| > |a_0 + \sum_{i=1}^{m-1}a_i e^{\ell_i x}|$. Thus $Q_\ell(x) \neq 0$.  

Similarly, if $\mathfrak{Re}(x)$ is sufficiently small, then for every $i$:  $$|e^{\ell_i x}| < |e^{\ell_1w}| < \frac{|a_0|}{M}$$ and thus $Q_\ell(x) \neq 0$. Since the above bounds only depend on $a_0, \ldots, a_m, \ell_1$ and $\ell_m - \ell_{m-1}$, we are done.  

\end{proof}

 \begin{definition} Given any $t = (w_0, \ell_0) \in  \BC \times \CS$, let $Q_t: \BC \to \BC$ be the function $Q_t(z) = Q(z + w_0, \ell_0)$.
 \end{definition}

\begin{definition} Fix $\omega \in \BC$. Let $\CI_\omega:  \BR^m \to \tup{Hol}(\BC)$ be the function sending $\ell$ to $Q_{(\omega, \ell)}$.
\end{definition}

\begin{definition} Given an entire function $f$, and $j \in \BN$, let $$\zeta_j(f) = \sup\{r| N_r(f) <j \}$$
where $N_r(f)$ is the number of zeroes of $f$ in $B(0,r)$, counted with multiplicity, and we use the convention that $\sup(\emptyset) = 0$. 
\end{definition}

\begin{definition} A holomorphic function is called $j$-separated if $\zeta_j(f) < \zeta_{j+1} (f)$. 
\end{definition}

\begin{observation}\label{obs3.1.6}
Let $t = (\omega, l) \in B \subset \CX$, and suppose that $B$ is connected. Suppose that there exists a $k$ such that for every $(w, \ell) \in B$, $\CI_\omega(\ell)$ is $k$-separated. Then $$\tup{diameter}(\pi_w B) \leq \sup_{\ell' \in \pi_\ell B} \zeta_k \big(\CI_\omega(\ell') \big)$$ 
\end{observation}

In order to study the above expressions, we first prove that the elements of $\{\CI_\omega(l)\}_{l \in \pi_\ell \beta(t,r)}$ belong to a well behaved class of functions. 

 \begin{definition} Let $E, R> 0$, and let $\spc$ be the set of all entire functions $f: \BC \to \BC$ with the following properties: 
 \begin{enumerate}
 \item For all $i \geq 0$: $|f^{(i)}(0)| \leq R^{i+1}$.
 \item $\exists 1 \leq i \leq m$ such that $|f^{(i)}(0)| > E$.
  
 \end{enumerate}
 Let $\spco \subset \spc$ be the subset we get by imposing the further condition $f(0) = 0$. 
 \end{definition}

\begin{prop}\label{Prop3.1.7} Let $A \subset \CS$ be relatively compact. There exists $E, R = E(A), R(A) > 0$ such that for any $t = (\omega, l) \in \CX_A$: $Q_t \in \spc$. 
\end{prop}

\begin{proof}

We note first that for every $i \in \BN$: $$Q_t^{(i)}(0) = \frac{\partial^i Q}{\partial w^i}(\omega, l) = \sum_{j=1}^m (l_j)^i a_i  e^{l_j \omega}$$

\nid Let $C_1>0$ be the constant provided by Lemma \ref{Lemma3.1.5} for the set $A$.  We have that $|\mathfrak{Re}(\omega)| < C$. Furthermore, since $A$ is relatively compact, there exists $D > 0$ such that $|l_j| < D$ for every $j$. Let $M = \max (|a_j|)$. We get:  $$|Q_t^{(i)}(0)| < mMe^{CD} D^i $$
Thus, we can choose a number $R$, satisfying the first part of the definition of $\spc$. We now turn to show that a number $E$ satisfying the second part of the definition can be chosen. \\

\nid Define the following functions on $\CX_A$:

$$\delta: \CX_A \to \BR^m: \delta(t) = \left(\begin{array}{c} Q_{t}' (0) \\ \vdots \\ Q_t^{(m)} (0) \end{array} \right) $$
$$v: \CX_A \to \BR^m: v(t) = v(\omega, l) = \left(\begin{array}{c} a_1 e^{l_1 \omega} \\ \vdots \\ a_m e^{l_m \omega} \end{array} \right) $$
$$X: \CX_A \to M_m(\BR): M(\omega, l)_{i,j} = {l_j}^i $$\\

\nid Notice that we have the relationship  $\delta(t) = X(t)v(t)$. \\

Fix $t = (w,\ell) \in \CX_A$. Let $=_\ell$ be the equivalence relation on $\{1, \ldots, m\}$ given by $j =_\ell k \iff \ell_j = \ell_k$. Given $\eta = \{j_1, \ldots, j_r \}$, a $=_\ell$ equivalence class, let $x_\eta = \sum_{j \in \eta} x_j$, where  $x_j: \BR^m \to \BR$ is the linear map that reads off the $j^{th}$ coordinate. Since $=_\ell$ is an equivalence relation, the set $\{x_\eta \}_\eta$, where $\eta$ varies over all equivalence classes, is linearly independent. Let $K = \cap_\eta \tup{Ker}(x_\eta)$, and let $r$ be the number of $=_\eta$ equivalence classes. We have that $\dim K = m-r$. 

We claim that $K = \tup{Ker} X(t)$. Indeed, it's obvious that $K \subset \tup{Ker}X(t)$. Furthermore, if $j_1, \ldots, j_{m-r}$ is a $=_\ell$ transversal, then the minor of $X(t)$ given by taking the $j_1, \ldots, j_{m-r}$ columns, and the first through $m-r$ rows is the product of $\tup{diag}(\ell_{j_1} \ldots, \ell_{j_{m-r}})$ with a Vandermonde matrix, and is thus invertible. Therefore, $\dim \tup{Ker}X(t) \leq m-r$, and we are done. \\

Let $u = (1, \ldots, 1) \in (\BR^m)^*$. Notice that for any $\xi \in K$, by the description above, we have that $u \cdot\xi = 0$. On the other hand, by the definition of $\CX$, we have that $u\cdot v(t) = -a_0$.  \\

By lemma  \ref{Lemma3.1.5}, we have that there exists constants $C_1, C_2 > 0$ such that $\forall t \in \CX_A$, $C_1< \| v(t)\| < C_2$. The function $\varphi: M_m(\BR) \to \BR$ given by $$\varphi(T) = \inf_{C_1 \leq \|\xi \| \leq C_2, u \cdot \xi = -a_0} \|T\xi\|$$
is continuous. By the above calculations, it does not vanish on $\overline{X(\CX_A)}$, and hence has a nonzero-minimum there, say $E'$. We have that for any $t \in \CX_A$, $\|\delta(t)\| = \|X(t) v(t) \| > E'$, and the result now follows by setting $E = \frac{1}{m}E'$.

\end{proof}

\begin{remark}\label{Rmk3.1.8} Note that the proof still works if we replace the set $\{C_1 \leq \|\xi \| \leq C_2, u \cdot \xi = -a_0 \}$ with the set $\{C_1 \leq \|\xi \| \leq C_2, |u \cdot \xi +a_0| < \epsilon\}$ for some fixed $0< \epsilon < |a_0|$. Thus, the result holds for all $t$ such that $|Q(t)| < \epsilon$. While we don't need this further generality in the proof of Theorem \ref{Theorem1}, we will use it later on in the proof of Theorem \ref{Theorem3}. 
\end{remark}

One important property of these spaces that we will use several times is the following: 

\begin{lemma}\label{Lemma3.1.9}The spaces $\spc$ and $\spco$ are compact in the compact-open topology. 
\end{lemma}
\begin{proof}
By the Cauchy-Hadamard test, given any $f \in \spc$, the radius of convergence of the Taylor series of $f$ about $0$ is $\infty$. Now, let $\{f_i\}_i \subset \spc$, and write: 
$$f_i = \sum_{j} c_{i,j} z^j $$
By compactness of the Hilbert cube, and by passing to a subsequence, we can assume that for each $j$, $c_{i,j}$ converges, say to $\gamma_{j}$. Let $\phi = \sum \gamma_j z^j$. Since $|\gamma_j| \leq \frac{R^{j+1}}{j!}$, we have that $f_i \to \phi$, in the compact open topology. The fact that $\phi \in \spc$ follows immediately from definitions. The compactness of $\spco$ is a direct corollary.

\end{proof}

\begin{definition} Define the following semi-norm on $\tup{Hol}(\BC)$:
$$\mathfrak{n}(\sum_i) a_i z^i  = \max_{0 \leq i \leq m} |a_i|$$   
Let $d_\mathfrak{n}$ be the associated semi-metric. We denote by $B_\mathfrak{n}(f,r)$ the ball of radius $r$ about $f$ in this semi-metric, intersected with $\spc$. 
\end{definition}

\begin{lemma}\label{Lemma3.1.10}There exists a number $\epsilon_0 > 0$, such that for any $g \in \spco$ there exists a $1 \leq j \leq m$ such that every $g' \in B_{\mathfrak{n}}(g,\epsilon_0)$ is $j$-separated. We say that $g$ is $(j,\epsilon_0)$ separated.
\end{lemma}

\begin{proof} Suppose not. Let $\{g_i\}_i \subset \spco$ be a sequence such that for each $i \in \BN$, and every $1 \leq j \leq m$,  $\exists \psi^j_i \in B_{\mathfrak{n}}(g_i, \frac{1}{i})$, such that $\psi_i^j$ is not $j$-separated. By compactness, and by passing to subsequences, we can assume that $\{g_i\}_i$ is convergent, say $g_i \to g$, and for all $j$, $\{\psi_i^j\}_i$ is convergent, say $\psi_i^j \to \psi^j$. For every $j$, we have that $d_{\mathfrak{n}}(\psi^j, g) = 0$. Since $\spco$ is closed, we must have that $g(0) =0$, and thus $\psi^j(0)=0$ for every $j$. Since $\psi^1$ is not $1$ separated, we must have that $(\psi^1)'(0)=0$ and thus $g'(0) = 0$. Since $\psi^2$ is not $2$ separated, we must have that $(\psi^2)''(0) = 0$, and thus $g''(0) =0$. Proceeding in this manner, we find that $g^{(j)}(0) = 0$ for all $1 \leq j \leq m$, which contradicts the fact that $g \in \spc$.
\end{proof}

\begin{lemma}\label{Lemma3.1.11} For all $j$, $\exists D_j > 0$ such that for any $(j,\epsilon_0)$ function $g \in \spco$:  $$\zeta_j\big(B_{\mathfrak{n}}(g,\epsilon_0) \big) \subset [0,D_j]$$
\end{lemma}

\begin{proof} Suppose not. Suppose that there exists a $j$ and two sequences $g_i$ and $\psi_i$ with $g_i$ a $(j,\epsilon_0)$ separated function, $\psi_i \in B_{\mathfrak{n}}(g_i, \epsilon_0)$, and $\zeta_j(\psi_i) > i$. By passing to subsequences, we can assume that $g_i \to g$, and $\psi_i \to \psi$. By uniform convergence and Rouche's Theorem, we have that $\zeta_j(\psi) = \infty$. Thus, $\psi$ is not $j$-separated. However, since $g_i \to g$, $\exists k > 0$ such that $\psi \in B_{\mathfrak{n}}(g_k, \epsilon_0)$, and thus $\psi$ must be $j$-separated. We have reached a contradiction.  

\end{proof}

\begin{lemma}\label{Lemma3.1.12} Fix a number $r > 0$. There exists $N = N(A,r) > 0$ such that for any $t = (\omega, l)\in \CX_A$, and any $\ell \in A$ such that $\|\ell - l\|< N(|\omega| + 1)$:
\begin{enumerate}
\item $\CI_\omega(\ell) \in \CB(\frac{1}{2}E, R)$.
\item $d_{\mathfrak{n}}\big(\CI_\omega(\ell), \CI_\omega(l)\big) < r$.
\end{enumerate}
\end{lemma}

\begin{proof} Fix $t = (\omega, l) \in \CX_A$. Given $(w, \ell)$, and $i \geq 0$, we calculate: $$\CI_\omega(\ell)^{(i)}(0) = \frac{\partial^i Q}{\partial w^i}(\omega, \ell) = \sum_{j=1}^m {\ell_j}^i e^{\ell_j \omega}$$
and that: 

$$\frac{\partial}{\partial \ell_j} \CI_\omega(\ell)^{(i)}(0) =  a_j e^{\omega \ell_j} \big(\omega {\ell_j}^i + 1  \big)$$

The fact that $\CI_\omega(\ell)$ satisfies the first condition of the definition of $\spc$ follows directly from the same proof that $Q_t$ satisfies this condition in Proposition \ref{Prop3.1.7}. For the other two claims in the statement of the lemma, we integrate $\frac{\partial}{\partial \ell_j} \CI_\omega(\ell)^{(i)}(0)$ and note that $e^{\omega \ell_j}$ and ${\ell_j}^i$ are bounded in $\CX_A$.  
\end{proof}

In the sequel, we replace $E$ with $\frac{1}{2}E$, and assume that if   $\|\ell - l\|< N(|\omega| + 1)$ then $\CI_\omega(\ell) \in \spc$. We can now prove Theorem \ref{Theorem1}. 

Let $\epsilon_0$ be the number provided by Lemma \ref{Lemma3.1.10}. Let $N = N(A, \epsilon_0)$  be the number provided by Lemma \ref{Lemma3.1.12}, and let $D_1, \ldots, D_m$ be the numbers provided by Lemma \ref{Lemma3.1.11}. Set $D = \max(D_1, \ldots, D_m)$. Define $r_1 = \frac{N}{|\omega| + 1}$. By Lemma \ref{Lemma3.1.12}, for any $t = (\omega, l) \in \CX_A$: $$\CI_\omega\big( \pi_\ell \beta(t,r_1) \big) \subset B_{\mathfrak{n}}(Q_t, \epsilon_0)$$

By Lemma \ref{Lemma3.1.10}, this set is $j$-separated for some $j$, and by Lemma \ref{Lemma3.1.11}, $\zeta_j\big[\CI_\omega\big( \pi_\ell \beta(t,r_1) \big) \big] \subset [0,D]$. Thus, by Observation \ref{obs3.1.6}, $$\tup{diameter } \pi_w\beta(t,r_1) < D$$

In particular, for any $t' = (\omega', l') \in \beta(t,r_1)$, we have that $|\omega'| < |\omega| + D$. Set $\delta_2 = \frac{N}{|\omega| + D + 1}$. By the above argument, for any $t' \in B(t, r_1)$, we have that $\tup{diameter } \pi_w \beta(t', \delta_2) < D$, and hence, if we denote $r_2 = r_1 + \delta_2$, we get: $\tup{diameter } \pi_w \beta(t, r_2) < 2D $. More generally, if we let $\delta_n = \frac{N}{|\omega| + (n-1)D + 1}$, and take $r_n = r_1 + \delta_2 + \ldots + \delta_n$ then: 
$$\tup{diameter } \pi_w \beta(t, r_n) < nD $$

Note that $\{r_n\}_n$ are the partial sums of a harmonic series. We thus have that $$r_n = N\big(\frac{1}{D} \ln n - \ln \frac{|\omega| + 1}{D}  \big) + O(1)$$
where the $O(1)$ above depends neither on $n$ nor on $t$. Thus, by exponentiating and plugging into the above inequality, we get that $\exists C_1, C_2, C_3 > 0$ such that $$\tup{diameter } \pi_w \beta(t,r_n) < C_1 (|\omega| + 1)^{C_2}e^{C_3 r_n} $$

Since $r_n \to \infty$, and $\tup{diameter } \pi_w \beta(t,r_n)$ is increasing in $r$, we get the desired result.

\end{proof}

\subsection{Proof of Theorem \ref{Theorem2}}\label{Section3.2}

\begin{proof}

\begin{definition} For $p = (p_1, \ldots, p_m) \in \CS$, let $\Lambda(p)$ be the set of all limit points of $\mathfrak{Re}(\CZ(q))$, as $q \to p$. 
\end{definition}

\begin{prop} \label{Prop3.2.1}
Given $p \in \CS$: 

$$\Lambda(p) = \{w \in \BR | \exists \xi_1, \ldots, \xi_m \in S^1, a_0 + \sum_{i=1}^m a_i \xi_i e^{p_i w} = 0 \} $$
\end{prop}

\begin{proof} Let $\omega \in \Lambda_p$. Then $\exists \{q_i \}_i \subset \CS, w_i \in \CZ(q_i)$ with $q_i \to p$, and $\mathfrak{Re}(w_i) \to \omega$. By compactness, we can assume that  for every $j$, there exists $\xi_j \in S^1$ such that $$e^{(q_i)_j \mathfrak{Im}(w_i)} \to \xi_j$$  By continuity, this  gives $$\Lambda(p) \subset \{w \in \BR | \exists \xi_1, \ldots, \xi_m \in S^1, a_0 + \sum_{i=1}^m a_i \xi_i e^{p_i w} = 0\} $$

Conversely, suppose that $w \in \BR$, and that $\xi_1, \ldots, \xi_m$ are chosen such that $a_0 + \sum_i a_i \xi_i e^{p_i w} = 0$. There exists a sequence $\{q_i\}_i \subset \CS$, with $q_i \to p$ such that for every $i$, the map $T_i: \BR \to (S^1)^m$ given by $T(x) = (e^{(q_i)_1 x}, \ldots, e^{(q_i)_m x} )$ has dense image in $(S^1)^m$. Indeed, the set of $q$'s having this property is co-null in $\CS$. In particular, we can choose $\{y_i\}_i \in \BR$ such that for any for all $j$: $$\lim_{i\to \infty} e^{(q_i)_j y_i} = \xi_i $$ Let $F: \BC^m \times \BC^m \times \BC \to \BC$ be the function:
$$F(A_1, \ldots, A_m, L_1, \ldots, L_m, w) = a_0 + \sum_{i=1}^m A_i e^{L_i w} $$

\nid The function $F$ is holomorphic, and  $F^{-1}(0)$ is an analytic set. We have that $F(a_1 \xi_1, \ldots, a_m \xi_m, p_1, \ldots, p_m, w) = 0$. Thus, by the open mapping theorem, for all $\epsilon > 0$, for all $q$ sufficiently close to $p$, and $\overline{A}$ sufficiently close to $(a_i \xi_i)_i$, $\exists w'$ with $F(\overline{A}, q, w') = 0$, and $|w' - w| < \epsilon$. Thus, for all sufficiently large $i$, we get $w_i'$ satisfying  $$F\big( (a_j e^{(q_i)_jy_i}  )_j, q_i, w_i' \big) = 0$$ and $|w_i' - w| < \epsilon$. Set $w_i = w_i' + y_i$, and send $\epsilon \to 0$, to see that $w \in \Lambda(p)$.
\end{proof}

\begin{remark}\label{Rmk3.2.2} It is simple to see from the above result that for any poly-exponential, $0 \in \Lambda(p)$ for every $p \in \CS$ or $0 \notin \Lambda(p)$, for every $p \in \CS$.
\end{remark}
Given $\overline{\xi} = (\xi_1, \ldots, \xi_m)$, and $w\in \BC$, denote by $G(\overline{\xi}, w) =a_0 + \sum_{i=1}^m a_i \xi_i e^{p_i w} $.  Note that:  $$\{w \in \BR | \exists \overline{\xi} \in (S^1)^m, G(\overline{\xi}, w) = 0 \} =  \{\mathfrak{Re}(w) |  \exists \overline{\xi} \in (S^1)^m, G(\overline{\xi}, w) = 0 \} $$
 The solutions of $G = 0$ are known to vary continuously in $\overline{\xi}$. This means that for a given $w \in \Lambda_p$ and $\overline{\xi} = (\xi_1, \ldots, \xi_m)$ such that $G(\overline{\xi},w) = 0$ , there is a continuous function $f_w: (S^1)^m \to \BC$ such that $f_w(\overline{\xi}) = w$, and $G(\overline{\eta}, f_w\big(\overline{\eta})\big) = 0$. Since for any $p$, the set $\Lambda(p)$ is the union over all $w \in \CZ(p)$ of the image of $(S^1)^m$ under $f_w$, we get the following.

\begin{corollary}\label{Cor3.2.3}
For any $p \in \CS$, the set $\Lambda(p)$ is an at most countable collection of closed intervals. 
\end{corollary}

\begin{observation}\label{Obsv3.2.4} Let $p \in \CS$, and let $\Lambda(p) = \bigcup_i [x_i,y_i]$. If, for some $i$, $x_i < \lambda_1(p) < y_i$ then $\lambda_1$ is not continuous at $p$. Similarly, if $x_i < \rho_1(p) < y_i$ then $\rho_1$ is not continuous at $p$. 
\end{observation}

\begin{lemma}\label{Lemma3.2.5} Let $p \in \CS$, and let $\Lambda(p) =\bigcup_i [x_i, y_i]$. Then for every $i$: $y_i - x_i > 0$. 
\end{lemma}

\begin{proof} Suppose $x_i = y_i = x$. Then, for any $\xi_1, \ldots, \xi_m \in S^1$:  $a_0 + \sum_{i=1}^m a_i \xi_i e^{p_i x} = 0$. In particular, this would be true for $\overline{\xi} = (1, \ldots, 1, 1)$ and for $\overline{\xi}' = (1, \ldots, 1, -1)$. Subtracting equations, we would get $2a_m e^{p_m x} = 0$, which is impossible. 
\end{proof}

\begin{definition} Let $p \in \CS$, and $j \in \BN$. Let $j^+ = \min\{k \in \BN|  k >j, \rho_k(p) \neq \rho_j(p) \}$. Similarly, let $j^- = \max\{k \in \BN| k < j, \rho_k(p) \neq \rho_j(p)  \}$.
\end{definition}

\begin{lemma} \label{Lemma3.2.6} Suppose $\Lambda(p) = \bigcup_i [x_i, y_i]$. Let $j \in \BN$. Suppose $\rho_j$ is continuous at $p$, and that $\rho_j(p) \in [x_i, y_i]$ for some $i$.
\begin{enumerate}
\item If $\rho_j(p) < y_i$, then $\rho_{j^-}$ is not continuous at $p$.
\item If $\rho_j(p) > x_i$, then $\rho_{j^+}$ is not continuos at $p$.
\end{enumerate}
\end{lemma}
\begin{proof} 
Suppose $\rho_j(p) < y_i$. If $\rho_{j^-}$ were continuous at $p$, we would have to have that $\Lambda(p) \cap \big(\rho_{j}(p) , \rho_{j^-}(p)\big) = \emptyset$, which is manifestly not the case. The same argument holds for the second claim in the lemma. 
\end{proof}

Note that a similar lemma holds for the $\lambda_j$'s. The first claim of Theorem \ref{Theorem2} follows directly from this lemma. \\

  It remains to prove the second claim.  Choose  $Q (w, \ell) =  e^{\ell_1 w} + \ldots + e^{\ell_m w} - 1$. Let $M = \CS - (\BR \cdot \BQ^m)$. Let $\ell \in \CM$.  Let $$\omega = \omega(\ell) = \sup\{ x \in \BR: \|e^{\ell_1x} \| \geq \|e^{\ell_2 x} \| + \ldots + |e^{\ell_m x}| + 1\}$$
Clearly $\omega$ is finite, and $e^{\ell_m \omega} - \sum_{i=2}^m e^{\ell_m \omega} - 1 = 0$. Furthermore, since $\ell \notin \BR \cdot \BQ^m$, $\nexists t \in \BR, z \in \BC$ such that $(e^{\ell_1 t}, \ldots, e^{\ell_m t}) = (z, \ldots, z, -z)$. Thus, $\nexists \overline{\xi} \in (S^1)^m \backslash (1, -1, \ldots, -1)$ such that $G(\overline{\xi}, \omega) = 0$. In particular, $Q(\omega, \ell) \neq 0$. Furthermore, by definition, we also have that if $w \in \CZ(\ell)$, then $\mathfrak{Re}(w) \leq \omega$. This shows that $\rho_1(p) \neq \sup \Lambda(p)$, and hence $\rho_1$ is not continuous at $\ell$. Notice that $\CM$ is dense and co-null to complete the proof. 

\begin{remark}\label{Rmk3.2.7} In this proof, we could have replaced $Q$ with $-1 + \sum_{i=1}^m \pm e^{\ell_i w}$, as long as there were at least two indices with a $+$ sign, and two with a $-$ sign. 
\end{remark}

\begin{remark}\label{Rmk3.2.8} The same proof shows that the poly-exponential $e^{\ell_1 (w)} - \sum_{i=2}^m e^{\ell_i w} - 1$ is Perron-Frobenius on the set $\CM \cap \{\ell: \ell_1 > \ell_2\ldots,\ell_m   \}$.
\end{remark}

\end{proof}

\subsection{Proof of Theorem \ref{Theorem3}}\label{Sec3.3}
\begin{proof}

Observe that: $\overline{Q}(w, \ell) =  e^{\ell_m w} Q(-w, \ell)$, and thus, for any $\ell \in \CS$, we have the equality $\CZ_+(Q, \ell) = -\CZ_-(\overline{Q}, \ell)$. Observe further that for any $\ell $ , and $r \in \BR$ we have that $\CZ(r \ell) = \frac{1}{r} \CZ(\ell)$. For $(p,q) \in \tup{Int}([0,\infty^m \times [0, \infty]^m)$, the result now follows directly from Theorem \ref{Theorem1}. Now, suppose $(p,q) \in \partial [0,\infty]^m \times [0, \infty]^m$

\begin{lemma}\label{Lemma3.3.1}
Let $\spc$ be the space defined in \ref{3.1.1}. For every $R > 0$, there exists $r < R$, and $\epsilon > 0$ such that for any $f \in \spc$, there exists $\rho < r$ such that 
$$\min_{|z| = \rho} |f(z)| > \epsilon $$

\end{lemma}

\begin{proof} Suppose not. Then there exists $R > 0$, such that for every $r < R$, $\epsilon > 0$, there exists a function $f$ such that for every $\rho < r$, $\min_{|z| = \rho} |f(z)| \leq \epsilon $. Fix a particular $0 < r < R$. Set $\epsilon_n = \frac{1}{n}$, and let $f_n$ be the functions provided by the above hypothesis with respect to $r, \epsilon_n$. Since $\spc$ is compact, we may assume that $f_n \to f$. Each circle of the form $\{|z| = \rho\}$ is compact. Thus, on in each such circle, $\exists z_\rho$ such that $f(z_\rho) = 0$. Thus, $f$ has a non-discrete set of zeroes, which is impossible since $f \neq 0$. 
\end{proof}

\begin{lemma}\label{Lemma3.3.2}

Let $Q_1: \BC \times \BR^r \to \BC$ and $Q_2: \BC \times \BR^s$ be poly-exponentials, such that $Q_1$ has a non-zero free term, but $Q_2$ does not. Let $Q: \BC \times \BR^s \times \BR^t \to \BC$ be given by $Q(w, \ell_1, \ell_2) = Q_1(w, \ell_1) + Q_2(w, \ell_2)$. Fix $l \in \BR^r_{> 0}$, and $M < 0$. Let $H_M = \{z \in \BC | \mathfrak{Re}(z) < M \}$. Then for any $0 < \xi < M$, there exist $\delta, D > 0$ such that whenever $|l_1 - l| < \delta$ and $l_2 \in \BR^s_{>0}$ satisfies that every component is greater than $D$:

$$\forall w \in \CZ(Q,l_1, l_2) \cap H_M, \exists \rho < \xi:  \CN(Q(\cdot, l_1, l_2), w, \rho) = \CN(Q_1(\cdot, l_1), w, \rho) $$
and:

$$\forall w \in \CZ(Q_1,l_1) \cap H_M, \exists \rho < \xi: \CN(Q(\cdot, l_1, l_2), w, \rho) = \CN(Q_1(\cdot, l_1), w, \rho) $$

Where $\CN(f,x,y)$ is the number of zeroes of $f$, counted with multiplicity in the ball of radius $y$ about $x$.

\end{lemma}

\begin{proof} Let $a_0$ be the free term of $Q_1$. Fix a number $0 < \gamma < |a_0|$. By remark \ref{Rmk3.1.8}, $\exists \delta_0 > 0$ and $E, R > 0$ such that whenever $|l - l_1| < \delta_0$, and $\omega \in H_M$ satisfies $|Q_1(\omega,l)| < \gamma$, then $Q_{(\omega,l )} \in \spc$. Let $r,\epsilon > 0$ be the numbers assigned to $\xi$ by lemma \ref{Lemma3.3.2}. For sufficiently large values of $D > 0$, we have that $Q(w, l_2) < \epsilon$, for any $w \in H_{M+\xi}$. Thus, if $Q_1(\omega, l_1) < \gamma$, the number of zeroes of $Q_1(\cdot, l_1)$ in $B(\omega, \rho)$ is the same as the number of zeroes of $Q(\cdot, l_1,l_2)$ in the same set, for some value of $\rho < r < \xi$, by Rouche's theorem. 

On the other hand, if $D$ is sufficiently large, then $|Q_2(w, l_2)| < \gamma$ for every $w \in H_M$. So if  $|Q_1(w, \ell_1)| > \gamma$,  then $Q(w, l_1, l_2) \neq 0$. Thus, for any $\rho$ sufficiently small, if $\CN(Q(\cdot, l_1, l_2), w, \rho) \neq 0$, then $\exists w' \in B(w, \rho)$ such that $|Q_1(w, l_1)| < \gamma$ and the result holds by the previous argument. 
 
\end{proof}

\begin{lemma}\label{Lemma3.3.3}

In the notation above, if  $\omega \in \CZ(p,q)$ then $\exists i$ such that $$\lim_{\iota(\ell \to (p,q))} \psi'_i(\ell)= \omega$$

\end{lemma}

\begin{proof}
Suppose, without loss of generality that $\omega \in \CH_-$. Let $J = \{j_1, \ldots, j_s \}$ be the set of coordinates such that $p_{j_i} = \infty$. Write $Q = Q_1 + Q_2$, where $Q_2 = \sum_{j \in J} a_j e^{\ell_j w}$. Set $M = \frac{\mathfrak{Re}(\omega)}{2}$.  By Lemma \ref{Lemma3.3.2}, for any $\xi < M$,  there exists a neighborhood $U$ of $(p,q)$, such that if $\iota(\ell) \in U$, then the number of $i$'s, counted with multiplicity, such that $\psi_i'(\ell) \in B(\omega, \xi)$ is a non-zero constant on $\overline{U}$. By continuity, this implies that the \emph{set} of $i$'s such that $\psi_i'(\ell) \in B(\omega, \xi)$ is constant on $\overline{U}$. Sending $\xi \to 0$ gives the result. 
\end{proof}

If $\psi$ is a root of $Q$, and $\lim_{\iota(\ell) \to (p,q)} \psi'(\ell) = \omega$, and $\omega \neq [0]$, then  by continuity, $\omega \in \CZ(p,q)$. Thus, to conclude the proof of Theorem \ref{Theorem3}, it suffices to prove the following lemma. 

\begin{lemma}\label{Lemma3.3.4}
If $\psi$ is a root of $Q$ then $\lim_{\iota(\ell) \to (p,q)} \psi'(\ell)$ exists. 
\end{lemma}

\begin{proof} 
Fix $M < 0$, and choose $0 < \xi < M$. Let $U$ be the neighborhood of $(p,q)$ provided by Lemma \ref{Lemma3.3.2}, with respect to $M, \xi$. If $\exists \ell \in \CS, \iota(\ell) \in U$ such that $\psi'(\ell) = \omega \in H_M$, then by Lemma \ref{Lemma3.3.2}, $\exists \omega' \in \CZ(Q_1, \ell) \cap B(\omega, \xi) = \CZ(p,q) \cap B(\omega, \xi)$, where $Q_1$ is the poly-exponential defined in the proof of Lemma \ref{Lemma3.3.3}. By choosing $\xi$ sufficiently small, we can assume that $\CZ (p,q)$ has a unique zero in this set. By the same argument as the one in the proof of the lemma, $\psi'(\ell') \in B(\omega', \xi)$, for every $\ell' \in \iota^{-1}(U)$. Sending $\xi \to 0$, we get that $\psi'(\ell) \to \omega'$, as $\iota(\ell) \to (p,q)$. A similar argument holds for the set $-H_M$. Thus, if there is no $\omega' \neq [0]$ such that $\psi' \to \omega'$, we must have that for every $M$, there is a neighborhood $U = U(M)$,  of $(p,q)$ such that $\psi'(\ell) \notin H_M \cup -H_M$ for every $\ell \in \iota^{-1}(U_M)$. But this implies that $\psi' \to [0]$.

\end{proof}

\nid Theorem \ref{Theorem3} is a direct consequence of Lemmas \ref{Lemma3.3.3} and \ref{Lemma3.3.4}.  
\end{proof}

\section{Examples}\label{Sec4}
\subsection{Example 1} \label{Sec4.1}

$$ \includegraphics[width=100mm]{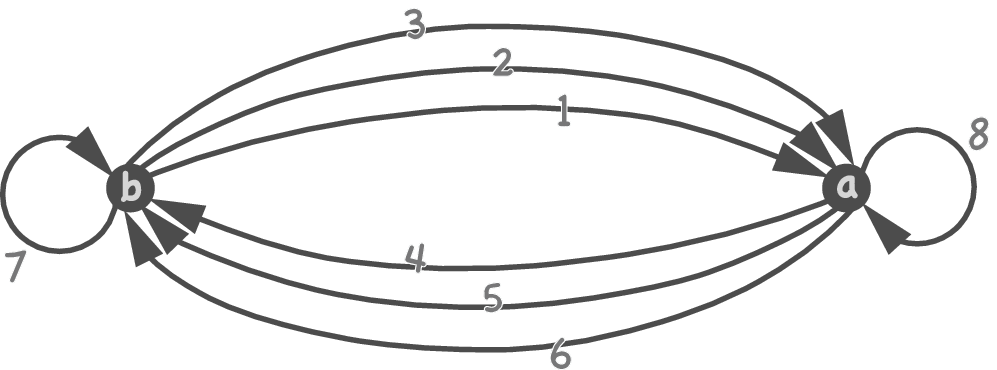} $$

Let $\Gamma$ be the graph pictured above. The Perron polynomial $P(t,\ell)$ is equal to $$\det \left(\begin{array}{c c} 1- t^{\ell_7} & - t^{\ell_1} - t^{\ell_2} - t^{\ell_3} \\ -t^{\ell_4 } - t^{\ell_5} - t^{\ell_6} & 1 - t^{\ell_8} \end{array}\right)$$ 

\nid Calculating its determinant we have: 

$$P(t,\ell) = t^{\ell_8 + \ell_7} - \sum_{i = 1}^3 \sum_{j = 4}^6 t^{\ell_i + \ell_j} - t^{\ell_7} - t^{\ell_8} +1 $$

Replacing $t^x$ with $e^{x w}$, we get the corresponding poly-exponential, which we call $Q = Q(w, \ell)$. Let $\ell$ be given by $$\ell_1 = \ldots = \ell_6 = 1, \ell_7 = \ell_8 = 2$$ 
and let $\ell'$ be given by: 
$$\ell'_1 = \ell_2 = \ell'_3 = 4, \ell_4 = \ell_5 = \ell_6 = 5, \ell_7 = \ell_8 = 6 $$

\nid We have that $P(t, \ell) = t^4 - 11t^2 + 1$, and $P(t, \ell') = t^{12} - 9t^9 - 2t^6 + 1$. The solutions to $P(t, \ell) = 0$ satisfy ${|t|}^2 = \frac{11 \pm \sqrt{117}}{2}$. Note that one of these absolute values is greater than one, and one is less. The solutions to $P(t, \ell') = 0$ satisfy that ${|t|}^3 \approx 9.215, 0.507, 0.421$. 

Let $R$ be the poly-exponential $R(w, l) = e^{l_1 w} - 9e^{l_2 w } + 2 e^{l_3 w} + 1$. It is simple to see that with respect to the poly-exponential $R$, for any $p \in \CS \subset \BR^3$, we have that $0 \notin \Lambda(p)$. 

The points $\ell$ and $\ell'$ can be connected by a path in the subspace $V = \{l | l_1 = l_2 = l_3, l_4 = l_5 = l_6, l_7 = l_8\}$. By the above remark, we see that in this subspace, there are no solutions of $Q = 0$ that have a real part of $0$. Thus, by Theorem \ref{Theorem1}, there exists a root $\psi$ of $Q$ such that $|\exp\big(\psi(\ell) \big)| = (\frac{11 + \sqrt{117}}{2})^{\frac{1}{2}}$, and $|\exp\big(\psi(\ell') \big)| \approx {9.215}^{\frac{1}{3}}$, and furthermore, any root $\psi$ of $Q$ satisfying $|\exp\big(\psi(\ell) \big)| = (\frac{11 + \sqrt{117}}{2})^{\frac{1}{2}}$ also satisfies $|\exp\big(\psi(\ell') \big)| \approx {9.215}^{\frac{1}{3}}$. 

\nid Using Theorem \ref{Theorem1}, we also get the existence of a root $\psi_1$ such that $|\exp\big(\psi_1(\ell) \big)| = (\frac{11 - \sqrt{117}}{2})^{\frac{1}{2}}$, and $|\exp\big(\psi_1(\ell') \big)| \approx {0.507}^{\frac{1}{3}}$, and a root $\psi_2$ such that $|\exp\big(\psi_2(\ell) \big)| = (\frac{11 - \sqrt{117}}{2})^{\frac{1}{2}}$ and $|\exp\big(\psi_2(\ell') \big)| \approx {0.421}^{\frac{1}{3}}$.Furthermore, any root of $Q$ satisfies one of the above descriptions.

 \subsection{Example 2} \label{Sec4.2}
 $$ \includegraphics[width=120mm]{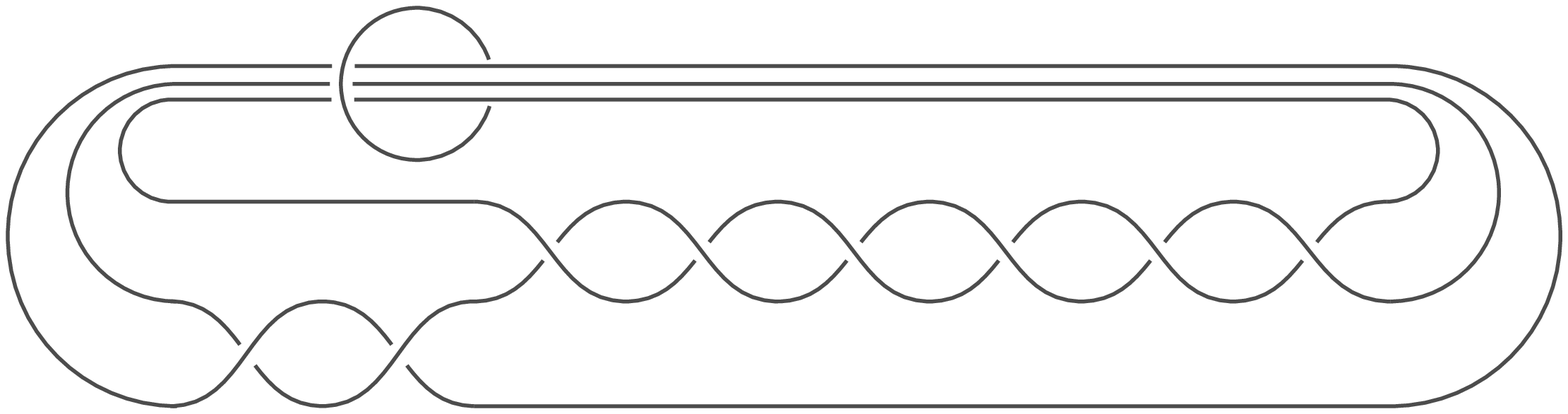} $$

 Let $L$ be a regular neighborhood of the link pictured above, and let $M = S^3 - L$. The manifold $M$ is compact, hyperbolic, and fibers over the circle. There is one fibration where the fiber is the thrice punctured disk, and the monodromy is given by the braid $\sigma_1^2 \sigma_2^{-6}$, where $\sigma_i$ is the half Dehn-twist that interchanges the $i^{th}$ and $(i+1)^{th}$ punctures. In \cite{McMullenP}, McMullen calculated the Tecihm\"uller polynomial of a fibered face of this manifold. We write it in a notation more amenable to our uses, expanding and multiplying out by an expression to give a constant term of $1$.  \\
 
 \nid We get:
 
 $$P(t, a,b,c,u) = t^{2u-a+2b+3c} + t^{u - a - b} + t^{u+ 3b + 3c} + t^{u - a + 3b + 4c} + t^{u + 3b + 4c}$$
 $$- \sum_{i=1}^4 \big[t^{u - a +(i-1)b + ic} + t^{u + (i-1)b + ic} + t^{u - a + (i-1)b + (i-1)c} + t^{u + (i-1)b + (i-1)c} \big]  + 1$$
 
 \nid where in $\CF$ the leading exponent is $2u - a + 2b + 3c $. It is simple to check that the point $(a,b,c,u) = (-1, 3, 1, 3) \in \partial \BR_+\CF$. Indeed, at this point $t^{2u - a + 2b + 3c} = t^{u-a+ 3b + 3c}$, but there are arbitrarily close points where $t^{2u - a + 2b + 3c}$ is the leading term. Consider the curve $\gamma(x) = (-1,3,1,3) - x(-1, 4, 5 ,0)$. As $x \to 0$, this path approaches the boundary of the fibered face, and thus the dilatation goes to $\infty$ along $\gamma$. We will calculate the rate at which it goes to $\infty$ as well as finding the rates of all other roots that go to $\infty$ along $\gamma$.
 
 Let $Q: \BC \times \BR_{>0}^{13}$ be the poly-exponential $Q(w, \ell) = e^{\ell_{13} w} - \sum_{i=1}^{12} e^{\ell_i w} + 1$. By ordering the monomials in $P(t,a,b,c,u)$ such that $t^{2u - a + 2b + 3c}$ is the last one, we can think of the polynomial $P(\cdot, a, b, c, u)$ as being $Q(\cdot, \ell)$ for some $\ell$. Precomposing this identification with $\gamma$, we get a curve $\ell(x): [0,1] \to \BR_{>0}^{13}$ such that $\ell(x) \to \partial \CS$ as $x \to 0$.  By \cite{McMullenP}, the polynomial $P$ is palindromic. Thus, we will not lose information by restricting ourselves to $\CZ_+(\ell)$. Calculating, we get that $\min(\bar{\ell}(x)) = 23x$. We have that  $\lim_{x \to 0} \iota(\ell(x)) = (p,q)$. As stated above, it is enough to calculate $\CZ(q)$.
 
At $(-1,3,1,3)$, the highest exponent in $P$ is $t^{20}$. There are only two monomials whose exponents approach $20$ as $x\to 0$. These are $t^{2u - a + 2b + 3c}$ and $-t^{u - a + 3b + 3c}$. Thus, we get that $q$ has only two coordinates that are not equal to $\infty$. Calculating gives us that these coordinates are $1$ (corresponding to $t^{2u - a + 2b + 3c}$), and $\frac{26}{23}$ (corresponding to $-t^{u - a + 3b + 3c}$).  Solving the equation: $$-t^{\frac{26}{23}} + t + 1 = 0$$
and taking the absolute values of the solutions, we see that the solution set contains $13$ distinct absolute values, $8$ of which are less than $1$. Call them $\gamma_1, \ldots, \gamma_8$. The smallest of these, $\gamma_1$, corresponds to the solution $\approx 0.972069^{\frac{1}{6}}$. 
Thus, recalling that $\Lambda$ is the function that assigns to a cohomology class the $\log$ of its dilatation, we get that, as $x \to 0$: 

$$ \frac{1}{23 \gamma_1x} \Lambda(\gamma(x))^{-1} \to 1$$
giving the growth rate of the dilatation along this path. Furthermore, for any other root $\psi$, we either have that $[\psi]$ is bounded along the path, or $$ \frac{1}{23 \gamma_i x} |\exp\big(\pm \psi(\ell(x)) \big)|^{-1} \to 1 $$
for $2 \leq i \leq 8$, where all possible $i$'s occur as growth rates for roots.

\subsection{Example 3}\label{Sec4.3}

$$ \includegraphics[width=80mm]{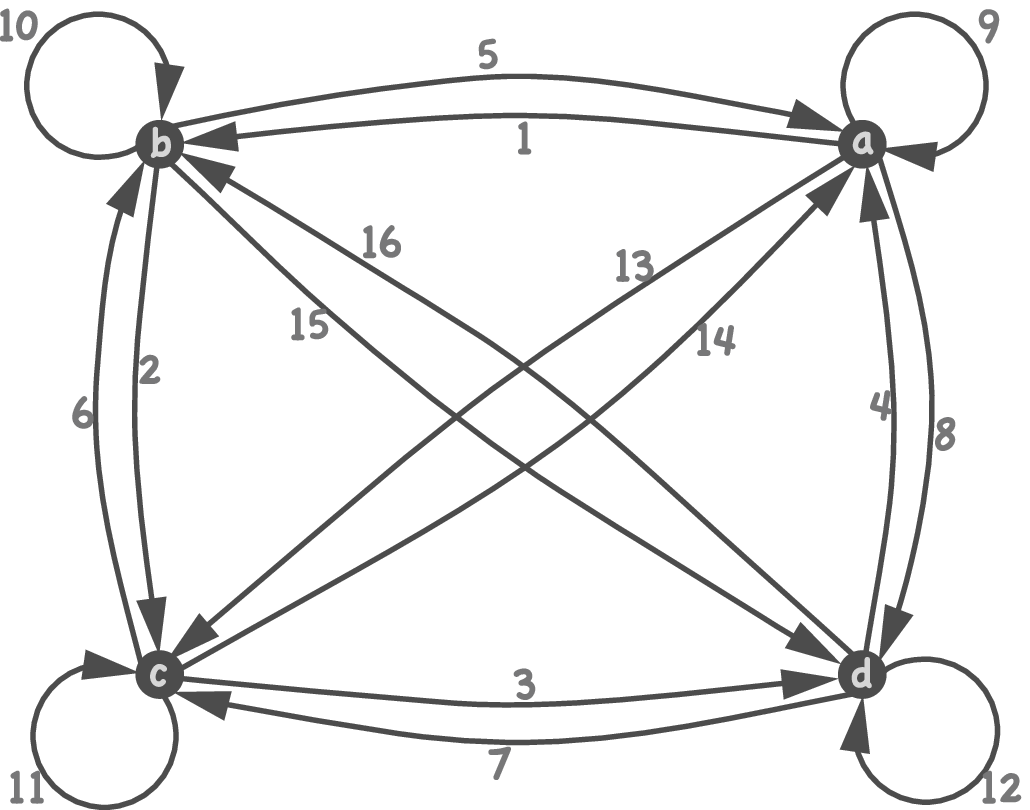} $$

Let $\Gamma$ be the graph pictured above, and let $T$ be the automorphism that rotates the graph by $\frac{\pi}{2}$. There is a single $T$ orbit of vertices of $\Gamma$. Let vertex $a$ be a representative in this orbit. Suppose we want to calculate $P_{T,-1}$. By the proof of Proposition \ref{Proposition5}, since there is only one orbit, we can take: 

$$P_{T, -1}  = 1 - t^{\ell_1} + t^{\ell_{9}} + t^{\ell_{11}} - t^{\ell_{8}}$$

By the proof of the second part of \ref{Theorem2}, we see that there is a co-null set $\CE \subset \CS$ on which $\lambda_1$ cannot be extended continuously for this poly-exponential.

\end{document}